\documentclass[preprint,12pt]{elsarticle}
\usepackage{graphics}
\usepackage{color}
\definecolor{rouge}{rgb}{1,0,0}
\definecolor{bleu}{rgb}{0,0,1}
\definecolor{vert}{rgb}{0,0.5,0}

\usepackage{graphicx}
\usepackage{cases}
\usepackage{amssymb}
\usepackage{amsthm}
\usepackage{amsfonts}
\newtheorem{definition}{Definition}
\newtheorem{proposition}{Proposition}
\newtheorem{corollary}{Corollary}
\newtheorem{lemma}{Lemma}
\newtheorem{theorem}{Theorem}

\journal{Journal of Differential Equations}

\newcommand{\cc}{\gamma}
\newcommand{\dis}{\displaystyle}
\newcommand{\dist}{\mbox{\bf  dist}}
\def\HH{{\cal{H}}}

\begin{document}

\begin{frontmatter}

\title{Stability of a critical nonlinear \\neutral delay differential equation}

\author[JAD,INRIA]{S. Junca\corref{cor1}}
\ead{junca@unice.fr}
\author[LMA]{B. Lombard\corref{cor2}}
\ead{lombard@lma.cnrs-mrs.fr}

\cortext[cor2]{Corresponding author}

\address[JAD]{Laboratoire J.A. Dieudonn\'e, UMR 7351 CNRS, 
          Universit\'e de Nice Sophia-Antipolis, Parc Valrose, 06108 Nice Cedex 02, France}
\address[INRIA]{Team Coffee,
INRIA Sophia Antipolis Mediterran\'ee, 
2004 route des Lucioles - BP 93,
06902 Sophia Antipolis Cedex, France}
\address[LMA]{Laboratoire de M\'ecanique et d'Acoustique, UPR 7051 CNRS, 31 chemin Joseph Aiguier, 13402 Marseille, France}

\begin{abstract}  
This work deals with a scalar nonlinear neutral delay differential equation issued from the study of wave propagation. A critical value of the coefficients is considered, where only few results are known. The difficulty follows from the fact that the spectrum of the linear operator is asymptotically closed to the imaginary axis. An analysis based on the energy method provides new results about the asymptotic stability of the constant and periodic solutions. A complete analysis of the stability diagram is given. Lastly, existence of periodic solutions is discussed, involving a Diophantine condition on the period.
\end{abstract}

\begin{keyword}
neutral delay differential equations \sep energy method \sep stability diagram \sep periodic solutions \sep small divisors \sep no exponential stability
%% MSC codes here, in the form: \MSC code \sep code
%% or \MSC[2008] code \sep code (2000 is the default)
\MSC 34K03 \sep 34K40 %theorie generale / neutral equations 
\end{keyword}

\end{frontmatter}

% \linenumbers

\section{Introduction}\label{SecIntro}

In this work, we consider the  following nonlinear neutral delay differential equation (NDDE) and we look for  the solution 
$y \in H^1_{loc}((-1,\, + \infty), \mathbb{R}) $,
\begin{equation}
\left\{
\begin{array}{l}
\displaystyle
y^{'}(t)+c\,y^{'}(t-1)+f(y(t))+g(y(t-1))=s(t), \quad t>0 , \\ % \quad t\notin\mathbb{N},\\
[8pt]
\displaystyle
y(t)=y_0(t)\in H^1((-1,\,0), \mathbb{R}), \quad -1<t<0,
\end{array}
\right.
\label{ToyModel}
\end{equation}
under the following assumptions:
\begin{subnumcases}{\label{TMhyp}}
\displaystyle
c=\pm 1,\label{Hyp1}\\
[8pt]
\displaystyle
f\in C^1(\mathbb{R},\,\mathbb{R}), \quad f^{'}>0, \quad f(0)=0,\label{TMHyp2}\\
[8pt]
\displaystyle
|g|\leq \cc\,|f|,\, \mbox{ with } 0 \leq \cc<1, \label{TMHyp3}\\
[8pt]
\displaystyle
s(t+T)=s(t),\quad T=\frac{\textstyle 2\,\pi}{\textstyle \omega}>0. \label{TMHyp4}
\end{subnumcases}
In some cases, additional assumptions will be required:
\begin{subnumcases}{\label{H1H2}}
\displaystyle
|g^{'}|\leq \cc\,|f^{'}|,\,\mbox{ with } 0 \leq \cc<1,\label{H1}\\
[8pt]
\displaystyle \lim_{y\rightarrow \pm \infty}f(y)=\pm \infty.\label{H2}
\end{subnumcases}
Equation (\ref{ToyModel}) is typically issued from hyperbolic partial differential equations with nonlinear boundary conditions \cite{Hale03,Erneux09}. A closely related system of neutral equations has been derived, for instance, in the case of elastic wave propagation across two nonlinear cracks \cite{JuLo12}: $y$ is then the dilatation of the crack, the shift 1 is the normalized travel time between the cracks, $f$ and $g$ denote the nonlinear contact law, and $s$ is the $T$-periodic excitation.

A natural issue when dealing with (\ref{ToyModel}) is to prove existence, uniqueness and stability of periodic solutions. Articles and reference books \cite{Lopes75,Erbe95,Das97,Parhi01,Gyori91,Gopalsamy92,KR12} usually address this question by considering a linear version of (\ref{ToyModel}) with $|c|<1$. The critical value $|c|=1$ raises technical difficulties: the spectrum of the linearized operator is asymptotically closed to the left of the imaginary axis, hence exponential stability is lost and small divisors may be encountered.

An old    contribution to the case $|c|=1$ is done in \cite{Gromova67}. In the linear case, algebraic rate of convergence is proven. In the nonlinear case with $c=-1$, algebraic convergence is also proven in this reference, but assuming small $C^1$ data. In both cases, the main  tools involved are asymptotic expansions of characteristic roots, Laplace transforms and function series.
Another recent contribution for linear  degenerate retarded differential systems is given in \cite{Zhou09}.

The aim of our article is to push forwards the analysis of (\ref{ToyModel}) in this critical case, especially to remove (when possible) the assumption of small solutions, and also to consider less regular initial data. A strategy based on energy analysis is followed. In the linear case, a stability analysis shows that the energy method is optimal. When the stability condition is violated, the existence of an essential instability is proven \cite{Hale03,KE}.

The paper is organized as follows. Section \ref{SecNonlinear} treats the full nonlinear NDDE. The stability of the zero solution in the homogeneous case is proven. Generalization to the non-homogeneous case with constant forcing is done. Non-constant periodic forcing is investigated in the case of small solutions. In section \ref{SecLinear}, new results are given in the linear case, concerning the stability diagram and the asymptotic rate of convergence. In section \ref{SecExistence}, the existence of periodic solutions is addressed. For particular values of the period, existence is proven whatever the amplitude of the source. More generally, existence is obtained under a Diophantine condition. This latter condition is satisfied if the period of the forcing is rational: $T=p\,/\,q$; if $p$ is odd, we also exhibit a smoothing effect. Lastly, conclusion is drawn in section \ref{SecConclu}, and some future lines of research are proposed. 

%---------------------------------------------------------------------
%---------------------------------------------------------------------

\section{Nonlinear NDDE}\label{SecNonlinear}

\subsection{Stability in the homogeneous case}\label{SecNonlinear0}

Let us introduce basic functions and inequalities useful to state a main result:
\begin{equation}
F(y)=\int_0^y f(z)\,dz,\quad G(y)=\int_0^y g(z)\,dz,\quad H(y)=2\,\left(F(y)-c\,G(y)\right),
\label{FGH}
\end{equation}
with $f$, $g$, $\cc$ and $c$ defined in (\ref{ToyModel})-(\ref{TMhyp}).

\begin{lemma}
For all $y$, $F$ and $H$ satisfy the inequality
\begin{equation}
0\leq 2\,(1-\cc)\,F(y)\leq H(y)\leq 2\,(1+\cc)\,F(y).
\label{FH}
\end{equation}
\label{LemmaFH}
\end{lemma}

\begin{proof}
Let assume $y>0$. From (\ref{TMhyp}) and (\ref{FGH}), it follows
$$
|G(y)|\leq \int_0^y |g(z)|\,dz\leq \int_0^y \cc\,|f(z)|\,dz=\int_0^y \cc\,f(z)\,dz=
\cc\,F(y).
$$
Similarly, if $y<z<0$ then $|f(z)|=-f(z)$ and 
$$
|G(y)| \leq \int_y^0 |g(z)|\,dz\leq \int_y^0 \cc\,|f(z)|\,dz=-\int_0^y (-\cc\,f(z))\,dz=\cc\,F(y).
$$
Consequently, one gets $0\leq |G| \leq \cc\,F$ for all $y$. Using (\ref{FGH}) and $c=\pm 1$ concludes the proof.
\end{proof}

\begin{theorem}[Asymptotic stability of the zero solution]
Let $y$ be the solution of (\ref{ToyModel}) without source term: $s=0$. For all $t>0$, one defines
\begin{equation}
E(t)=\int_{t-1}^t \left(y^{'}(\tau)+f(y(\tau))\right)^2\,d\tau\geq 0.
\label{NRJ}
\end{equation}
Then one obtains
\begin{equation}  
\sup_{t>0} E(t)+\sup_{t>0} F(y(t))+\int_0^{+\infty}f(y(t))^2\,d t<+\infty,
\label{inE0}
\end{equation}
where $F$ is defined in (\ref{FGH}). It follows the asymptotic stability of the origin:
\begin{equation}
\lim_{t\rightarrow +\infty}y(t)=0.
\end{equation}
\label{ThZero}
\end{theorem}
The stability is proven by an energy method. Notice that the inequality (\ref{inE0}) allows convergence towards zero with an algebraic rate. Indeed there is not exponential stability of the equilibrium $y=0$, as we will see for instance in the linear case in section \ref{SecLinear}.

\begin{proof}
When $s=0$, the NDDE (\ref{ToyModel}) writes
\begin{equation}
y^{'}(t)+f(y(t))=-c\,\left(y^{'}(t-1)+c\,g(y(t-1))\right),
\end{equation}
where $c=\pm 1$. Taking the square yields
\begin{equation}
\left(y^{'}(t)+f(y(t))\right)^2=\left(y^{'}(t-1)+c\,g(y(t-1))\right)^2.
\label{ProofTh0-1}
\end{equation}
Based on (\ref{FGH}), the right-hand side of (\ref{ProofTh0-1}) at $t-1$ is rewritten 
\begin{equation}
\begin{array}{lll}
\left(y^{'}+c\,g\right)^2 &=& \left(\left(y^{'}+f\right)-\left(f-c\,g\right)\right)^2,\\
[6pt]
&=& \left(y^{'}+f\right)^2+\left(f-c\,g\right)^2-2\,\left(y^{'}+f\right)\,\left(f-c\,g\right),\\
[6pt]
&=& \left(y^{'}+f\right)^2+\left(f-c\,g\right)^2-2\,y^{'}\left(f-c\,g\right)-2\,f\left(f-c\,g\right),\\
[6pt]
&=& \displaystyle \left(y^{'}+f\right)^2+\left(f-c\,g\right)^2-\frac{\textstyle d}{\textstyle dt}\,H(y)-2\,f\left(f-c\,g\right),\\
[6pt]
&=& \displaystyle \left(y^{'}+f\right)^2+\left(f-c\,g\right)\,\left(f-c\,g-2\,f\right)-\frac{\textstyle d}{\textstyle dt}\,H(y),\\
&=& \displaystyle \left(y^{'}+f\right)^2-\left(f^2-g^2\right)-\frac{\textstyle d}{\textstyle dt}\,H(y).
\end{array}
\end{equation}
The latter equation is injected into (\ref{ProofTh0-1}), which yields
\begin{equation}
\begin{array}{l}
\displaystyle
\left(y^{'}(t)+f(y(t))\right)^2+\left(f^2(y(t-1))-g^2(y(t-1))\right)^2\\
[6pt]
\displaystyle
+\frac{\textstyle d}{\textstyle dt}\,H(y(t-1))=\left(y^{'}(t-1)+f(y(t-1))\right)^2.
\end{array}
\label{ProofTh02}
\end{equation}
Integrating (\ref{ProofTh02}) over $[0,\,t]$, we get for all $t>1$
\begin{equation}
\begin{array}{l}
\displaystyle
\int_0^t\left(y^{'}(\tau)+f(y(\tau))\right)^2\,d\tau+\int_{-1}^{t-1}\left(f^2(y(\tau))-g^2(y(\tau))\right)\,d\tau\\
[6pt]
\displaystyle
+H(y(t-1))-H(y(-1))=\int_{-1}^{t-1}\left(y^{'}(\tau)+f(y(\tau))\right)^2\,d\tau.
\end{array}
\label{ProofTh03}
\end{equation}
Equation (\ref{ProofTh03}) is simplified into the energy equality
\begin{equation}
\begin{array}{l}
\displaystyle
\int_{t-1}^t \left(y^{'}(\tau)+f(y(\tau))\right)^2\,d\tau+\int_0^{t-1} \left(f^2(y(\tau))-g^2(y(\tau))\right)\,d\tau+H(y(t-1))\\
[6pt]
\displaystyle =\int_{-1}^0 \left(y^{'}_0(\tau)+f(y_0(\tau))\right)^2\,d\tau+H(y_0(-1))-\int_{-1}^0 \left(f^2(y_0(\tau))-g^2(y_0(\tau))\right)\,d\tau.
\end{array}
\label{ProofTh0-3}
\end{equation}
From lemma \ref{LemmaFH} and equation (\ref{ProofTh0-3}), one deduces the energy inequality for all $\tau>1$
\begin{equation}
\begin{array}{l}
\displaystyle
\int_{t-1}^t \left(y^{'}(\tau)+f(y(\tau))\right)^2\,d\tau+(1-\cc^2)\,\int_0^{t-1} f^2(y(\tau))\,d\tau+2\,(1-\cc)\,F(y(t-1))\\
[10pt]
\displaystyle 
\leq \int_{-1}^0 \left(y^{'}_0(\tau)+f(y_0(\tau))\right)^2\,d\tau+2\,(1+\cc)\,F(y_0(-1))-(1-\cc^2)\int_{-1}^0 f^2(y_0(\tau))\,d\tau\\
[10pt]
\displaystyle
=C_0.
\end{array}
\end{equation}
For all $t>1$, it follows three inequalities:
\begin{equation}
F(y(t))\leq\frac{\textstyle C_0}{\textstyle 2\,(1-\cc)},\quad \int_0^{+\infty}f^2(y(t))\,dt \leq \frac{\textstyle C_0}{\textstyle 1-\cc^2},\quad E(t)\leq C_0.
\label{ProofTh0-4}
\end{equation}
The first inequality in (\ref{ProofTh0-4}) implies $y\in L^{\infty}(0,\,+\infty)$. On any bounded set $K$, there exists $d>0$ such that $d\,|y|\leq|f(y)|\leq|y|\,/\,d$ on $K$. Using the second inequality in (\ref{ProofTh0-4}) and the boundedness of $y$, one obtains $y\in L^2(0,+\infty)$. Lastly, the third inequality ensures that $E$ is bounded; since $f(y)\in L^2(0,\,+\infty)$, one deduces that $y^{'}$ is uniformly bounded in $L^2(T-1,\,T)$, and hence $y$ is bounded in the H\"{o}lder space $C^{0,1/2}(0,\,+\infty)$. It proves that $y$ tends towards 0.
\end{proof}

Four remarks are raised by theorem \ref{ThZero}:
\begin{itemize}
\item if $c=0$ and $\cc=0$, then the NDDE is a simple ODE. Our computation provides  for all $T >0$: 
$2F(y(T)) + \int_0^T (y'(t))^2 + f^2(y(t)) dt = 2F(y(0))$  which implies the global  stability of $y=0$. 
Moreover  the stability of $y= 0$ is exponential since $f'(0) > 0$;
\item in the usual case $|c|<1$, then a stronger energy estimate can be obtained by our proof:
$y\in H^1(0,\,+\infty)$, and the stability is exponential;
\item in the present case $|c|=1$, then $y\in L^2\cap L^{\infty}$ and $y$ is uniformly bounded in $H^1(\tau-1,\,\tau)$ for all $\tau$. As we will see in section \ref{SecLinear} for the linear critical NDDE, the stability is not exponential;
\item if $|c|=1$ and $\cc=1$, then the stability is obtained, but the asymptotic stability can be lost.
\end{itemize}

%---------------------------------------------------------------------

\subsection{Stability with constant source term}\label{SecNonlinearCte}

In this subsection, we consider a non-homogeneous NDDE (\ref{ToyModel}) with constant source term $D$:
\begin{equation}
\left\{
\begin{array}{l}
\displaystyle
y^{'}(t)+c\,y^{'}(t-1)+f(y(t))+g(y(t-1))=D, \quad t>0,\quad t\notin\mathbb{N},\\
[8pt]
\displaystyle
y(t)=y_0(t)\in H^1(-1,\,0), \quad -1<t<0,
\end{array}
\right.
\label{ToyModel-D}
\end{equation}
with assumptions (\ref{TMhyp}). For the sake of clarity, we analyze successively $g=0$ and $g\neq 0$.

\begin{corollary}[case $g=0$]
Let $y$ be the solution of (\ref{ToyModel-D}). Three cases occur:
\begin{enumerate}
\item if $D=f(d)$, then d is the unique globally attractive solution;\vspace{0.2cm}
\item if $D>\sup f$ (or $D<\inf f$), then $\displaystyle\lim_{t\rightarrow +\infty}y(t)=+\infty$ (or $-\infty$);
\item if $D=\sup f$ (or $D=\inf f$), then there is no convergence towards a constant solution.
\end{enumerate}
\label{CoroD-f}
\end{corollary}

\begin{proof}
We consider successively the three cases.

\noindent
\underline{Case 1: $D=f(d)$}. This case occurs if $\displaystyle \lim_{y\rightarrow \pm\infty} f(y)=\pm\infty$. Injecting
$$
z=y-d,\qquad f_d(z)=f(d+z)-f(d)
$$
into (\ref{ToyModel-D}) yields the homogeneous NDDE
$$
z^{'}(t)+c\,z^{'}(t-1)+f^{'}_d(z(t))=0.
$$
Since $f_d(0)=0$ and $f^{'}_d>0$, the assumptions of theorem \ref{ThZero} are satisfied: $\displaystyle\lim_{t\rightarrow +\infty}z(t)=0$, and thus $y$ tends asymptotically towards $d$.\\
\noindent
\underline{Case 2: $D>\sup f$ (or $D<\inf f$)}. To fix the minds, let us assume $\sup f <D<+\infty$, and introduce $\delta=D-\sup f>0$. 
\begin{itemize}
\item If $c=-1$, then $y^{'}(t)-y^{'}(t-1)=D-f(y)\geq \delta$, so that $y^{'}(n+\tau)\geq y^{'}_0(\tau)+n\,\delta$, with $\tau\in[-1,\,0]$. As a consequence, $\displaystyle\lim_{t\rightarrow +\infty}y^{'}(t)=+\infty$, and hence $\displaystyle\lim_{ t\rightarrow +\infty}y(t)=+\infty$.
\item If $c=+1$, then $y^{'}(t)+y^{'}(t-1)\geq \delta$. Integration on $[\tau-1,\,\tau]$ yields $y(\tau)-y(\tau-2)\geq \delta$, and once again $\displaystyle\lim_{t\rightarrow +\infty}y(t)=+\infty$.
\end{itemize}
\noindent
\underline{Case 3: $D=\sup f$ (or $D=\inf f$)}. Let us assume $y(t) \rightarrow d$ when $t \rightarrow + \infty$. Let $y_n(s)$ be $y(n+s)$ for $s \in [0,1]$. We have $y_n(.) \rightarrow d$ in $L^\infty([0,1])$ and $y'_n \rightarrow 0$ in the sens of distribution. The NDDE can be rewritten as follows: $y'_n+c\,y'_{n-1}+f(y_n)=D$. Taking the weak limit, we get $f(d)=D$ which is impossible since $f<D$. Consequently, $y(.)$ cannot converge towards a constant. 

Using the same argument, we can state that $y(.)$ cannot converge towards a periodic continuous solution. More precisely, let us assume that $y_n(s) \rightarrow d(s)$ where $d(.)$ is continuous.   
Writing $y_n(0)=y(n)=y((n-1)+1)=y_{n-1}(1)$ yields $d(0)=d(1)$ i.e. $d$ is a 1-periodic continuous function. Taking the weak limit in the NDDE yields a differential equation for $d(.)$:
$(1+c)\,d'(t)+f(d(t))=D$. 
\begin{itemize}
\item If $c=-1$ then $f(d(t))=D$, which is impossible. 
\item If $c=1$ then $d \in C^1([0,1])$ by the equation. Let $s_0$ be a maximizer of $d(.)$ on the compact set $[0,1]$, then $d'(s_0)=0$ and, by the differential equation, $f(d(s_0))=D$ which is again impossible.
\end{itemize} 
As a consequence, $y(.)$ cannot converge towards a periodic continuous solution.
\end{proof}

Now, let us examine the NDDE (\ref{ToyModel-D}) with all the terms.

\begin{corollary}[case $g\neq 0$]
Under the assumptions (\ref{H1H2}), the unique $d$ such as $f(d)+g(d)=D$ is the unique globally attractive solution of (\ref{ToyModel-D}).
\label{CoroD-fg}
\end{corollary}

\begin{proof}
Assumptions (\ref{H1H2}) easily imply that there exists a unique $d$ satisfying $f(d)+g(d)=D$. Injecting
$$
z=y-d,\qquad f_d(z)=f(d+z)-f(d),\qquad g_d(z)=g(d+z)-g(d)
$$
into (\ref{ToyModel-D}) yields the homogeneous NDDE
$$
z^{'}(t)+c\,z^{'}(t-1)+f^{'}_d(z(t))+g^{'}_d(z(t-1))=0.
$$
Inequality (\ref{H1}) implies that $\cc\,|f_d|\geq|g_d|$. Since $f_d(0)=0$ and $f^{'}_d>0$, all the assumptions of theorem \ref{ThZero} are satisfied, and hence $z\rightarrow 0$ asymptotically.
\end{proof}

Without assumptions (\ref{H1H2}) in corollary \ref{CoroD-fg}, one may encounter more complex situations, with 2 or more solutions. Note that these assumptions are satisfied in the physically-relevant situation examined in \cite{JuLo12}.

%---------------------------------------------------------------------

\subsection{Stability with periodic source}\label{SecNonlinearPeriod}

In this section, we consider the non-homogeneous NDDE (\ref{ToyModel}) with the periodic source term $s$, but in the particular case $g=0$. Existence of periodic solutions to (\ref{ToyModel})-(\ref{TMhyp}) is examined in section \ref{SecExistence}. Here, we assume that a periodic solution $P(t)$  exists and  the difference with another solution $y(t)$ is bounded for all time .  We focus on the convergence of $y(.)$ towards $P(.)$.  The convergence is obtained  under the condition (\ref{HypoStab}).

\begin{proposition}[Conditional asymptotic stability of a periodic solution]
Let $y$ be the solution of (\ref{ToyModel}) with $f\in C^2$ and $g=0$, and $P$ be a periodic solution. Let $I=\mbox{conv}\left(y[-1,\,+\infty[,\,P(\mathbb{R}) \right)$ be the convex hull of the two sets. If $I$is bounded and the condition
\begin{equation} 
2\,\sup|P^{'}|<\displaystyle \inf_{w\,\in\,I,\,d\,\in\,P(\mathbb{R})}\frac{\textstyle \left(f(w)-f(d)\right)^2}{\textstyle \left|f(w)-f(d)-(w-d)\,f^{'}(d)\right|}
\label{HypoStab}
\end{equation}
is satisfied, then
$\displaystyle \lim_{t\rightarrow +\infty}(y(t)-P(t))=0$.
\label{PropPeriode}
\end{proposition}

We try to extend the proof of Theorem \ref{ThZero} to more complex case, but the extension is very limited. It means that if $P$ is a small $C^1$ periodic solution and $y(.)$ is a small stable perturbation, then $y(.)$ converges towards $P$. In other words, stability and condition (\ref{HypoStab}) imply asymptotic stability. The difficulty follows from the occurrence of small divisors; see section \ref{SecLinear}.

\begin{proof}
Injecting $z(t)=y(t)-P(t)$ in (\ref{ToyModel}) yields
\begin{equation}
z^{'}(t)+c\,z^{'}(t-1)+\Delta(t,\,z(t))=0,
\label{ProofPeriod0}
\end{equation}
with
\begin{equation}
\Delta(t,\,z)=f(z+P(t))-f(P(t))=\int_0^zf^{'}(u+P(t))\,du=z\int_0^1 f^{'}(s\,z+P(t))\,ds.
\label{ProofPeriod1}
\end{equation}
Let us introduce
$$
K(t,\,z)=\int_0^z\Delta(t,\,u)\,du.
$$
$K$ is non-negative and convex with respect to $z$, and it satisfies
\begin{equation}
\begin{array}{lll}
\displaystyle
\frac{\textstyle d}{\textstyle dt}\,K(t,\,z)&=& \displaystyle z^{'}\,\Delta(t,\,z)+P^{'}(t)\int_0^z\left(f^{'}(u+P(t))-f^{'}(P(t))\right)\,du,\\
&=& \displaystyle z^{'}\,\Delta(t,\,z)+P^{'}(t)\left(f(z+P(t))-f(P(t))-z\,f^{'}(P(t))  \right),\\
&=& \displaystyle z^{'}\,\Delta(t,\,z)+P^{'}(t)\,z^2\int_0^1 f^{''}(z\,u+P(t))\,(1-u)\,du,\\
&=& \displaystyle z^{'}\,\Delta(t,\,z)+P^{'}(t)\,z^2\,R(t,\,z).
\end{array}
\label{ProofPeriod2}
\end{equation}
Now, we mimic the proof of theorem \ref{ThZero}. Based on (\ref{ProofPeriod0}) and (\ref{ProofPeriod2}), one deduces
\begin{equation}
z^{'2}(t)+\Delta^2(t,\,z)+2\,\frac{\textstyle d}{\textstyle dt}\,K(t,\,z)-2\,P^{'}(t)\,z^2(t)\,R(t,\,z))=z^{'2}(t-1).
\label{ProofPeriod3}
\end{equation}
Since $P$ is periodic and non constant, the term $P^{'}(t)\,z^2(t)\,R(t,\,z))$ has no fixed sign and there is no hope to control (\ref{ProofPeriod3}). It motivates the condition (\ref{HypoStab}) for small derivative. Based on (\ref{ProofPeriod1}) and (\ref{ProofPeriod2}), this condition is equivalent to say that there exists $0\leq \tau<1$ such that
\begin{equation}
\begin{array}{l}
\displaystyle \hspace{0.5cm} 2\,|P^{'}(t)|\left|f(z+P)-f(P)-z\,f^{'}(P)\right|\leq \tau\,\left(f(z+P)-f(P)\right)^2,\\
\Leftrightarrow \displaystyle 2\,|P^{'}(t)|\left|\int_0^1f^{''}(u\,z+P(t))\,(1-u)\,du\right|\leq \tau\,\left( z \int_0^1 f^{'}(z\,  u+P(t))\,du \right)^2,\\
[15pt]
\Leftrightarrow \displaystyle 2\,P^{'}(t)\,z^2(t)\,R(t,\,z))\leq \tau\,\Delta^2(t,\,z).
\end{array}
\label{ProofPeriod4}
\end{equation}
Integrating (\ref{ProofPeriod3}) over $[0,\,\tau]$ and using the inequality (\ref{ProofPeriod4}) yields
\begin{eqnarray}
\int_{\tau-1}^\tau z^{'2}(t)\,dt+(1-\tau)\int_0^\tau\Delta^2(t,\,z)\,dt+2\,K(\tau,\,z(\tau))
\label{ProofPeriod5}
\\ \leq \int_{-1}^0z^{'2}_0(t)\,dt+2\,K(0,\,z_0(0)).  
\end{eqnarray}
The first term in the left-hand side of (\ref{ProofPeriod5}) implies that $z$ is uniformly continuous, whereas the third term implies that $z$ is bounded. As in the proof of theorem \ref{ThZero}, the second term allows to conclude.
\end{proof}

The stability of the solution to (\ref{ToyModel}) with a non-constant source term  remains an open question. Such a $L^{\infty}$ bound would imply that the interval $I$ in proposition \ref{PropPeriode} is bounded.

%---------------------------------------------------------------------
%---------------------------------------------------------------------

\section{Linear NDDE}\label{SecLinear}

In this section, we consider the linear homogeneous NDDE issued from (\ref{ToyModel})
\begin{equation}
y^{'}(t)+c\,y^{'}(t-1)+a\,y(t)+b\,y(t-1)=0.
\label{NDDElinear}
\end{equation}

%---------------------------------------------------------------------

\subsection{Stability diagram}\label{SecLinearDiagram}

The stability of 0 in (\ref{NDDElinear}) is analyzed in terms of $a=f^{'}$ and $b=g^{'}$. The cases $|c|<1$ and $|c|>1$ are well known in the literature, yielding to asymptotic stability and instability, respectively \cite{Hale03}. We focus therefore on the critical case $|c|=1$ for linear scalar differential-difference equations discussed in \cite{Bellman63}, p. 191. In this book, it is said that the stability is too much intricate, so the model would never be used by engineers !      
 
This analysis has two interests. First, it highlights the optimality of the energy analysis performed in theorem \ref{ThZero}: as in the nonlinear case, where $|g|<f$ and $f^{'}>0$, asymptotic stability is proven iff $|b|<a$. Second, the lost of exponential stability is explicitly exhibited. 

The stability analysis will involve the roots $\lambda$ of the characteristic equation
\begin{equation}
h(\lambda)\equiv h(\lambda;\,a,\,b)=\lambda\,\left(e^\lambda+ c \right)+a\,e^\lambda+b=0
\label{Characteristic}
\end{equation}
obtained by injecting a particular solution $y(t)=e^{\lambda\,t}$ into (\ref{NDDElinear}). 

\begin{proposition}
The characteristic roots of (\ref{Characteristic}) satisfy the following properties:
\begin{enumerate}
\item the roots are located in a vertical strip $\lambda_{\inf}<\Re(\lambda)<\lambda_{\sup}$, where $\lambda_{\inf}$ and $\lambda_{\sup}$ are real constants depending on $(a,b,c)$;
\item the set of roots is countable infinite, yielding a sequence $\lambda_n$;
\item for large values of $|\lambda|$, one gets the asymptotic value $\lambda_n\sim i\,\Omega(n)$, where 
\begin{equation}
\Omega(n)=\left(2\,n+(1+c)\,/\,2\right)\,\pi;
\label{OmegaN}
\end{equation}    
\item the sign of $\Re(\lambda_n)$ is constant on each quarter plans separated by $a=\pm b$;
\item the sign of $\Re(\lambda_n)$ changes when performing a symmetry relative to the line $a+c\,b=0$.
\end{enumerate}
\label{PropCharacteristic}
\end{proposition}

\begin{proof} 
We prove successively the properties.\\
\underline{Properties 1 and 2}. These are classical results \cite{Hale03,JuLo12}, not repeated here. 

\noindent
\underline{Property 3}. If $\lambda$ is a root of (\ref{Characteristic}) and $|\lambda|\rightarrow+\infty$, then
\begin{equation}
\displaystyle e^\lambda=-\frac{\textstyle c\,\lambda+b}{\textstyle \lambda+a}\sim-c,
\label{ExpOmegaN}
\end{equation}
and hence $\lambda\equiv \lambda_n\sim i\,\Omega(n)$ in (\ref{OmegaN}).

\noindent
\underline{Property 4}. Let us consider a imaginary characteristic root $\lambda=i\,w$. Injecting this root into (\ref{Characteristic}) leads to
\begin{subnumcases}{\label{ProofStabAB}}
\displaystyle
b+a\,\cos w - w\,\sin w=0,\label{ProofStabAB1}\\
[8pt]
\displaystyle
a\,\sin w+w\,\left(c+\cos w\right)=0.\label{ProofStabAB2}
\end{subnumcases}
If $\sin w=0$, then $w=n\,\pi$ and the following alternative occurs: if $n=0$ then $a+b=0$, else $a-c\,b=0$. If $\sin w\neq 0$, then one injects $a\,\sin w=-w\,(c+\cos w)$ into (\ref{ProofStabAB}), which yields
\begin{equation}
\begin{array}{lll}
\displaystyle
b\,\sin w&=&\displaystyle-a\,\sin w \,\cos w+w\,\sin^2 w=c\,w\,\cos w+ w\,\cos^2 w+w\,\sin^2 w,\\
[8pt]
&=&
\displaystyle w\,\left(1+c\,\cos w\right),\\
[8pt]
&=& -a\,c\,\sin w,
\displaystyle
\end{array}
\end{equation}
and hence $a+c\,b=0$. By Rouch\'e's theorem, the roots of (\ref{Characteristic}) depend continuously on $a$ and $b$. Consequently, the sign of $\Re(\lambda_n)$ is constant on each quarter plans separated by the lines $a=\pm b$. 

\noindent
\underline{Property 5}. From the characteristic equation (\ref{Characteristic}), one deduces
\begin{equation}
\begin{array}{lll}
h(\lambda;\,a,\,b) &=& e^{\lambda}\,\left(c\,\lambda\,\left(e^{-\lambda}+c\right)+c^2\left(b\,e^{-\lambda}+a\right)\right), \\
[8pt]
&=& -c\,e^{\lambda}\,\left(-\lambda\,\left(e^{-\lambda}+c\right)-c\,b\,e^{-\lambda}-c\,a\right),\\
[8pt]
&=& -c\,e^{\lambda}\,h(-\lambda;\,-c\,b,\,-c\,a),
\end{array}
\end{equation}
which concludes the proof.
\end{proof}

%----------------------------------------

For the sake of clarity, we will distinguish the cases $c=+1$ and $c=-1$ in (\ref{NDDElinear}). The case $c=+1$ is examined in detail in the next theorem, whereas the modifications induced by $c=-1$ are given further with a short proof. 

\begin{theorem}[Stability diagram for $c=+1$]
We determine the stability of $y=0$ in (\ref{NDDElinear}) in terms of $(a,b)$. Only one open quadrant is asymptotically stable, the other three being unstable:
\begin{itemize}
\item $a>|b|$: asymptotic stability ($\forall n,\,\Re(\lambda_n)<0$);
\item $b<-|a|$: instability ($\forall n,\, \Re(\lambda_n)>0$);
\item $b>|a|$: instability ($\Re(\lambda_n)>0$, except one root $\Re(\lambda)<0$);
\item $a<-|b|$: instability ($\Re(\lambda_n)<0$, except one root $\Re(\lambda)>0$).
\end{itemize}
On the four edges, one has:
\begin{itemize}
\item $a=|b|$: stability ($\Re(\lambda_n)=0$, except one root $\lambda=-a<0$ if $b>0$);
\item $ -a= b>0$: stability (the spectrum belongs to $i \mathbb{R}$,  $ \forall n, \; \Re(\lambda_n)=0$);
\item $a=b<0$: instability: $\Re(\lambda_n)=0$, except one positive root $\lambda=a >0$.
\end{itemize}
Finally  $a=0,\; b=0$ is a stable case.
\label{ThStabDiagP1}
\end{theorem} 

\begin{proof}
Based on property 4 of Proposition \ref{PropCharacteristic}, we analyze successively the signs of $\Re(\lambda)$ on each quarter plan.

\noindent
\underline{Quadrant 1: $a\geq|b|$}. The proof of theorem \ref{ThZero} is adapted to the case of complex solutions $y(t)=e^{\lambda\,t}$, $\lambda\in\mathbb{C}$. Instead of (\ref{ProofTh0-3}), one obtains
\begin{equation} \label{nrjC}
\begin{array}{l}
\displaystyle
\int_{t-1}^t\left|y^{'}(\tau)+a\,y(\tau)\right|^2\,d\tau+(a-b)\,\left|y(t-1)\right|^2+(a^2-b^2)\int_0^{t-1}|y(\tau)|^2\,d\tau\\
\displaystyle
=\int_{-1}^0\left|y^{'}_0(\tau)+a\,y_0(\tau)\right|^2\,d\tau+\int_{-1}^0|y_0(\tau)|^2\,d\tau+(a-b)\,|y_0(-1)|^2.
\end{array}
\label{ProofThStab}
\end{equation}
This equality yields stability results when $a-b\geq0$ and $a^2-b^2\geq 0$, ie $a\geq|b|$. If $a>|b|$, then asymptotic stability follows directly from (\ref{ProofThStab}), as in the proof of theorem \ref{ThZero}. If $a=|b|$, the energy equality (\ref{nrjC}) becomes
\begin{equation} \label{nrjCbord}
\int_{t-1}^t\left|y^{'}(\tau)+|b|\,y(\tau)\right|^2\,d\tau + (|b|-b)\,\left|y(t-1)\right|^2=C_0. 
\end{equation}
Three cases are distinguished:
\begin{enumerate}
\item $a=b=0$.  The NDDE becomes $y'(t)+y'(t-1)=0$ and $h(s)=0$ iff $s(e^\lambda+1)=0$. The solutions of the characteristic equation (\ref{Characteristic}) are also known explicitly: $\lambda=0$  and $\lambda=i\,(2\,k+1)\,\pi$, $k\in\mathbb{Z}$ with multiplicity 1, which proves the stability. Notice that the energy equality (\ref{nrjCbord}) only means that all $\lambda_n \in  i \mathbb{R}$.
\item $a=b>0$. The NDDE becomes $z'(t)+a z(t) =0$ with $z(t)=y(t)+y(t-1)$. One therefore obtains $z(t)=C e^{-a t}$ and $y(t)=D^{-a t}+ w(t) $, where $w$ is a solution of the homogeneous difference equation $w(t)+w(t-1)=0$. The spectrum is given by $\lambda=-a$ and$\lambda=i\,(2\,k+1)\,\pi$, $k\in\mathbb{Z}$. All the roots have multiplicity 1, which proves the stability.
\item $0<a=-b$, energy yields 
$$ 
\int_{t-1}^t\left|y^{'}(\tau))\right|^2+a^2\,\left|y(\tau)\right|^2\,d\tau + a\,\left|y(t)\right|^2 + a\,\left|y(t-1)\right|^2=C_0, 
$$
which allows to conclude since $y(t)$ is bounded. Furthermore, this equality also means that each characteristic root is simple and is an imaginary number.
\end{enumerate} 

%----------------------------------------

\noindent
\underline{Quadrant 2: $b\leq -|a|$}. Based on property 5 of proposition \ref{PropCharacteristic}, the inner quadrant and its two edges are successively analyzed:
\begin{enumerate}
\item $b<-|a|$. From the case 1 of quadrant $a>|b|$, one deduces that $\Re(\lambda)>0$ when $b<-|a|$.
\item $b=a<0$. This line is obtained by symmetry of $b=a>0$. From the case 3 of quadrant $a>|b|$, it follows that all the roots are imaginary, except one root $\lambda=-a>0$: this edge is therefore unstable.
\item $b=-a<0$. Symmetry of roots relative to $a+b=0$ yields $\Re(\lambda)=0$: this edge is therefore stable.
\end{enumerate}

%-----------------------------------------

\noindent
\underline{Quadrant 3: $b\geq |a|$}. When $a=b=0$, the spectrum of (\ref{Characteristic}) is known: $\lambda=0$ and $\lambda=\lambda_m=i\,m=i\,(2\,k+1)\,\pi$. In the case $\lambda=0$, the implicit function theorem yields $\lambda^{'}(b)=-1/2$ and hence $\lambda(b)=-b/2+{\cal O}(b^2)$. Consequently, one obtains $\Re(\lambda)<0$ since $b>0$.

In the cases $\lambda=\lambda_m$, the implicit function theorem yields $\lambda_m^{'}(b)=-i\,m$ and hence 
\begin{equation}
\lambda_m(b)=i\,m-\frac{\textstyle i}{\textstyle m}\,b+A\,b^2+{\cal O}(b^3).
\label{ProofThStabTFI}
\end{equation}
To determine $A$, we inject (\ref{ProofThStabTFI}) into (\ref{Characteristic}):
\begin{equation}
h(\lambda_m(b),\,0,\,b)=\left(i\,m\left(\frac{\textstyle 1}{\textstyle 2\,m^2}-A\right)+\frac{\textstyle 1}{\textstyle m^2}\right)\,b^2+{\cal O}(b^3)=0,
\end{equation}
which leads to $A=1/(2\,m^2)-i/m^3$. Injecting $A$ into (\ref{ProofThStabTFI}) yields
\begin{equation}
\lambda_m(b)=i\,m-\frac{\textstyle i}{\textstyle m}\,b+\left(\frac{\textstyle 1}{\textstyle 2\,m^2}-\frac{\textstyle i}{\textstyle m^3}\right)\,b^2+{\cal O}(b^3),
\end{equation}
which therefore proves $\Re(\lambda_m)>0$.

This analysis of spectrum is valid only on a neighborhood of $(a=0,\,b=0)$ with $b>0$. The property 4 of proposition \ref{PropCharacteristic} allows to extend the result to the whole quadrant.

%-----------------------------------------

\noindent
\underline{Quadrant 4: $a\leq -|b|$}. This case is obtained by performing a symmetry of roots on quadrant 3 relative to the line $a+b=0$. It follows that there exists a single root with $\Re(\lambda)>0$ and an infinite sequence of roots such that $\Re(\lambda_m)<0$.

\noindent
\underline{Edge $b=-a>0$}. Symmetry relative to $a+b=0$ implies that if $\lambda$ is a root of (\ref{Characteristic}), then $-\lambda$ is also a root. Moreover, 0 is a simple root if $b\neq 2$; otherwise, $h^{'}(0,\,-b,\,b)=2-b=0$ and 0 is of multiplicity 2.

Now, let us consider the points $M(a,\,b)$ with $-b<a<0$ and $O$ the origin. The segment $OM$ is outside the half-line $(-b,\,b)$ and belongs to the quadrant $b>|a|$. At $M$, $h(0,\,a,\,b)=b-|a|>0$ and $\displaystyle \lim_{x\rightarrow -\infty}h(x,\,a,\,b)=-\infty$: there exists a negative real root $\lambda_0(a,\,b)$ at $M$. Since there is only one root with negative real part on quadrant $b>|a|$ (see quadrant 3), $\lambda_0$ is unique. By symmetry, there exists also a unique positive real root in the quadrant $a<-|b|$. Consequently, the only root with negative real part on quadrant $b>|a|$ connects the only root with positive real part on quadrant $a<-|b|$ at $\lambda_0(-b,\,b)=0$. 

On the other hand, the other roots on quadrants $b>|a|$ and $a<-|b|$ satisfy $\Re(\lambda_n)>0$ and $\Re(\lambda_n)<0$, respectively. By Rouch\'e's theorem, they depend continuously on $(a,\,b)$. Consequently, $\Re(\lambda_n)=0$ on $(-b,\,b)$.
\end{proof}

Four remarks are raised by theorem \ref{ThStabDiagP1}:
\begin{enumerate}
\item The asymptotically stable region $a>|b|$ in the plane of parameters $(a,b)$ is exactly the region given by theorem \ref{ThZero}, i.e. the energy method is optimal for the linear case to get the asymptotically stable region;
\item $a=|b|$: the stability follows from the same arguments as in the case $a>|b|$: energy method: theorem \ref{ThZero}, and a study of the characteristic roots as in \cite{JuLo12}. Moreover, crossing this edge, all real parts become positive in $b<-|a|$. This kind of dramatic transition, where all real parts are negative on $a>|b|$ becomes positive on  $b<-|a|$, is an {\it essential instability} \cite{Hale03,Erneux09,KE};
\item $-a=b>0$: this is a stable case if and only if all roots are simple. But it is not always the case. For instance, the root $s=0$ is simple on this edge except $(a=-2,\,b=2)$.  In this last case, $s=0$ is of multiplicity 2 and the unbounded solution $y(t)=t$ occurs;
\item the central case $a=b=0$, at the boundary of all previous subdomains, is stable.
\end{enumerate}

%--------------------------------------

We treat briefly in a similar way the case $c=-1$. Stability of $y=0$ in (\ref{NDDElinear}) is studied in terms of $(a,b)$.

\begin{theorem}[Stability diagram for $c=-1$]
$\mbox{}$\\
\noindent
Only one open quadrant is asymptotically stable, the other are unstable:
\begin{itemize}
\item $a>|b|$: asymptotic stability ($\forall n,\,\Re(\lambda_n)<0$);
\item $b >|a|$: instability ($\forall n,\, \Re(\lambda_n)>0$);
\item $b<-|a|$: instability ($\Re(\lambda_n)>0$, except one root $\Re(\lambda)<0$);
\item $-a >|b|$: instability ($\Re(\lambda_n)<0$, except one root $\Re(\lambda)>0$).
\end{itemize}
On the four edges, one has: 
\begin{itemize}
\item $a=|b|$: stability ($\Re(\lambda_n)=0$, except one root $\lambda=-a<0$ if $b<0$);
\item $a= b>0$: stability ($\Re(\lambda_n)=0$);
\item $a= -b <0$: instability ($\Re(\lambda_n)=0$, except one positive root $\lambda=-a $).
\end{itemize}
At the origin, one has
\begin{itemize}
\item $a=b=0$: weak instability ($\Re(\lambda_n)=0$, except 0 with multiplicity 2).
\end{itemize}
\label{ThStabDiagM1}
\end{theorem} 

\begin{proof}
The proofs of the various cases are shortly sketched:
\begin{enumerate}
\item We begin by the energy equality for $c=-1$: 
\begin{equation} \label{nrjCm}
\begin{array}{l}
\displaystyle
\int_{t-1}^t\left|y^{'}(\tau)+a\,y(\tau)\right|^2\,d\tau+(a+b)\,\left|y(t-1)\right|^2+(a^2-b^2)\int_0^{t-1}|y(\tau)|^2\,d\tau\\
\displaystyle
=\int_{-1}^0\left|y^{'}_0(\tau)+a\,y_0(\tau)\right|^2\,d\tau+\int_{-1}^0|y_0(\tau)|^2\,d\tau+(a+b)\,|y_0(-1)|^2.
\end{array}. 
\end{equation}
This energy equality yields the asymptotic stability when $a>|b|$. 
\item When $a=-b$, the NDDE becomes $z'(t)+a\,z(t)=0$ with $z(t)=y(t)-y(t-1)$. The spectrum is therefore $-a$ and on the imaginary axis. When $a=0$, 0 is a root with multiplicity 2; when $a \neq 0$, the roots are simple.
\item The stability on $a=b>0$ is given by the energy equality.
\item The symmetry induced by property 5 of proposition \ref{PropCharacteristic} concludes the case $ b> |a|$ and $b=-a>0$.
\item Passing trough $a=1, b=-1$, we obtain the sign of the real part of roots in $ b < -|a|$ for instance, and by symmetry in $-a>|b|$.
\item On $a=b<0$, the real part of roots are null, as in the proof of theorem \ref{ThStabDiagM1}.
\end{enumerate}
\end{proof}

%---------------------------------------------------------------------

\subsection{Small divisors}\label{SecSD}

Now, we consider the linearized non-homogeneous NDDE issued from (\ref{ToyModel})
\begin{equation}
{\cal L}\,y = y^{'}(t)+c\,y^{'}(t-1)+a\,y(t)+b\,y(t-1)=s(t),
\label{NDDElinearS}
\end{equation}
where $s$ is a $T$-periodic function or an almost periodic function \cite{Corduneanu89}, and $c=\pm 1$. We assume $a>|b|$: as stated in theorems \ref{ThStabDiagP1} and \ref{ThStabDiagM1}, this is the only asymptotically stable case. Based on (\ref{Characteristic}), we introduce
${\cal H}=e^{- \lambda}\,h(\lambda)$ which satisfies
\begin{equation}
\begin{array}{l}
\displaystyle
{\cal L}(\exp(\lambda t))={\cal H}(\lambda)\,\exp(\lambda t),\\
[6pt]
\dis
{\cal H}(\lambda)=\lambda\,(1+c \exp(-\lambda))+a+b\,\exp(-\lambda)),
\label{H}
\end{array}
\end{equation}
and $|{\cal H}(i\,\omega)|=|h(i\,\omega)|$. Lastly, we denote by $S$ and $Y$ the mean values of $s$ and $y$ over one period: for instance, $Y(t)=\frac{1}{T}\int_{t-T}^t y(\tau)\,d\tau$.

If $S=0$, then elementary calculations yield ${\cal L}\,Y=0$. In this case, theorem \ref{ThZero} ensures that $\dis \lim_{t \rightarrow + \infty}Y(t)=0$. The mean value of the periodic solution - if it exists - is therefore asymptotically stable.

If the source is monochromatic $s(t)\equiv e^{i\,\omega\,t}$, then $y(t)= y_p(t)+z(t)$, where the unique periodic solution is $y_p(t)=e^{i\,\omega\,t}\,/\,{\cal H}(i\,\omega)$, and $z$ satisfies the homogeneous equation (\ref{NDDElinear}): ${\cal L}\,z =0$. Theorem \ref{ThZero} ensures that $\displaystyle \lim_{t\rightarrow +\infty} z(t)=0$. In other words, a monochromatic source provides a unique periodic solution which is globally asymptotically stable. 

Now, we consider a source with an infinite spectrum:
\begin{equation}
s(t)=\sum_{k\in \mathbb{Z}} s_k\,e^{i\,\omega_k\,t}.
\label{SFourier}
\end{equation}
A formal solution of (\ref{NDDElinearS}) is
\begin{equation}
y_p(t)=\sum_{k\in \mathbb{Z}}\frac{\textstyle s_k}{\textstyle \HH(i\,\omega_k)}\,e^{i\,\omega_k\,t}.
\label{YFourier}
\end{equation}
The denominator in (\ref{YFourier}) never vanishes on the imaginary axis (${\cal H}(i\,\omega) \neq 0$), but 
$$ 
\dis \liminf_{|\omega| \rightarrow \infty}|{\cal H}(i\,\omega)|=0,\qquad
\dis \limsup_{|\omega| \rightarrow \infty}|{\cal H}(i\,\omega)|=+\infty, 
$$ 
and hence various cases can occur. We illustrate some situations on particular values of the period of the forcing $T$:
\begin{itemize}
\item $c=+1$
\begin{itemize}
\item if $T=1$, then ${\cal H}(i\,\omega_k)={\cal H}(i\,2\,k\,\pi)=4\,i\,k\,\pi+a+b$, and hence $\displaystyle |{\cal H}(i\,\omega_k)|\rightarrow+ \infty$ as $k\,\rightarrow +\infty$. Consequently, a smoothing effect occurs: $s \in L^2$ leads to $y_p \in H^1$;
\item if $T=2$, then ${\cal H}(i\,\omega_k)={\cal H}(i\,k\,\pi)=i\,k\,\pi\,\left(1+(-1)^k\right)+a+(-1)^k\,b$ and hence ${\cal H}(i\,\omega_k)=a-b$ if $k$ is odd: no smoothing effect. 
\end{itemize}
\item $c=-1$
\begin{itemize}
\item if $T=1$, then ${\cal H}(i\,\omega_k)=a+b$: no smoothing effect;    
\item if $T=2$, then ${\cal H}(i\,\omega_k)=i\,k\,\pi\,\left(1-(-1)^k\right)+a+(-1)^k\,b$ and hence ${\cal H}(i\,\omega_k)=a+b$ if $k$ is even: no  smoothing effect.
\end{itemize} 
\end{itemize}
A worst behavior is expected if $\dis \lim_{|k|\rightarrow +\infty} {\cal H}(i\,\omega_k)=0$, where small divisors may occur. The goal of the next theorem is to determine the sequence of frequencies $\omega_k$ leading to this situation.

\begin{theorem}
We assume $|c|=1$ and $a>|b|$ in (\ref{Characteristic}). If $\displaystyle \lim_{|k|\rightarrow +\infty}h(i\,\omega_k)=0$, then there exists a function $\phi:\,\mathbb{Z}\rightarrow\mathbb{Z}$ and a real sequence $s_k$ such that
\begin{equation}
\begin{array}{l}
\displaystyle
\lim_{|k|\rightarrow +\infty}\left|\phi(k)\right|=+\infty,\qquad \lim_{|k|\rightarrow +\infty} s_k=0,\\
\displaystyle
\omega_k=\Omega(\phi(k))+  \frac{\textstyle d}{\textstyle \Omega(\phi(k))}  +  \frac{\textstyle s_k}{\textstyle \Omega(\phi(k))}
\label{SDomegaK}
\end{array}
\end{equation}
where $\Omega(k)= ( 2 k + (1+c)/2)\pi$ is defined in (\ref{OmegaN}), and 
\begin{equation} 
d=a-c\,b>0.
\label{CoefD}
\end{equation}
Moreover, the small divisor has the same modulus than 
\begin{equation}
h(i\,\omega_k)=-s_k+{\textstyle i}\,\frac{\textstyle a^2- b^2}{\textstyle 2\,\Omega(\phi(k))}\,
+{\cal O}\left(|s_k|+\frac{1}{|\Omega(\phi(k))|}\right)^2.
\label{SmallDiv}
\end{equation}
\label{ThSmallDiv}
\end{theorem}

\begin{proof}
Two steps are involved:

\noindent
\underline{Step 1: $h(i\,\omega_k)\neq 0$}. Under the hypothesis $a>|b|$, the property 1 of proposition \ref{PropCharacteristic} can be made more precise. The characteristic roots (\ref{Characteristic}) satisfy
\begin{equation}
(\lambda+a)\,e^\lambda=-(c\,\lambda+b).
\end{equation}
The equality $\lambda+a=0$ implies $b=c\,a$ which is impossible when $|b|<a$. As a consequence, one obtains
\begin{equation}
e^\lambda=-\frac{\textstyle c\,\lambda+b}{\textstyle \lambda+a}\,\Rightarrow \,\left|e^{\lambda}\right|=\frac{\textstyle }{\textstyle }<1,
\end{equation}
which proves that $\Re(\lambda)<0$. The characteristic function $h$ therefore never vanishes on the imaginary axis.

%-------------------------------

\noindent
\underline{Step 2: asymptotic expansion of $\omega_k$}. From (\ref{Characteristic}), it follows
\begin{equation}
\frac{\textstyle h(i\,\omega_k)}{\textstyle i\,\omega_k}=1+c\,e^{-i\,\omega_k}+{\cal O}\left(\frac{\textstyle 1}{\textstyle \omega_k}\right).
\end{equation}
By hypothesis, $h(i\,\omega_k)\rightarrow 0$ when $|\omega_k|\rightarrow +\infty$. It results that $1+c\,e^{-i\,\omega_k}\rightarrow 0$, and hence there exists $z_k=\Omega(\phi(k))$ and $r_k\rightarrow 0$ so that
\begin{equation}
\omega_k=z_k+r_k,
\label{Rk1}
\end{equation}
with $\Omega_k$ defined in (\ref{OmegaN}). Injecting (\ref{Rk1}) into (\ref{Characteristic}) and using $e^{\pm i\,z_k}=-c$, we obtain
\begin{equation}
\begin{array}{lll}
\displaystyle
h(i\,\omega_k) &=& \displaystyle i\,(z_k+r_k)\,\left(1-e^{-i\,r_k}\right)+a-c\,b\,e^{-i\,r_k},\\
[6pt]
&=& \displaystyle i\,(z_k+r_k)\,\left(i\,r_k+\frac{\textstyle r_k^2}{\textstyle 2}+{\cal O}(r_k^3)\right)+a-c\,b\,\left(1-i\,r_k+{\cal O}(r_k^2)\right),\\
[10pt]
&=& \displaystyle -z_k\,r_k+i\,c\,b\,r_k+a-c\,b+{\cal O}\left(z_k\,r_k^2\right)+{\cal O}\left(r_k^2\right).
\end{array}
\label{HomegaK}
\end{equation}
Since $h(i\,\omega_k)\rightarrow 0$, there exists $s_k\rightarrow 0$ such that
\begin{equation}
r_k=\frac{\textstyle d}{\textstyle z_k}+\frac{\textstyle s_k}{\textstyle z_k},
\label{Rk2}
\end{equation}
with $d$ defined in (\ref{CoefD}). Putting (\ref{Rk2}) into the second line of (\ref{HomegaK}) gives
$$ 
h(i\,\omega_k)=-s_k+\frac{\textstyle i}{\textstyle z_k}\,\left(\frac{\textstyle d^2}{\textstyle 2}+c\,b\,d\right)+{\cal O}\left(\frac{\textstyle 1}{\textstyle z_k^2}\right)+{\cal O}\left(\frac{\textstyle s_k}{\textstyle z_k}\right)+{\cal O}(s_k^2).
$$
Notice that $d^2/2+c\,b\,d=d\,(d/2+c\,b)=d\,((a-c\,b)/2+c\,b)=d\,(a+c\,b)/2=(a-c\,b)\,(a+c\,b)/2=( a^2-c^2\,b^2)/2>0$, which concludes the proof.
\end{proof}

Theorem \ref{ThSmallDiv} states that the small divisors problem occurs when the spectrum of the source $s$ approaches the asymptotic values of the characteristic roots in the way (\ref{SDomegaK}). If $s$ is $T$-periodic, then the solution of (\ref{NDDElinearS}) looses accordingly at most one derivative.

%---------------------------------------------------------------------
%---------------------------------------------------------------------

\section{Existence of periodic solutions}\label{SecExistence}

We consider again the full nonlinear non autonomous NDDE (\ref{ToyModel}): indeed, some results for the nonlinear case are based on the linear analysis performed along section \ref{SecLinear}. 
We denote by $H^1_{\sharp}$ the Sobolev space of T-periodic functions with derivatives in $L^2$.

Two cases are distinguished, depending on the period $T$ of the source. If $1\,/\,T \in \mathbb{N}$, then general results can be proved whatever the amplitude of the solution. For a much larger set of periods and for small sources, results can be obtained under a condition on the pulsation $\omega =\frac{2\pi}{T}$. This condition  is equivalent to a Diophantine condition on $T$. We finish the section by some remarks and open problems about this Diophantine condition. 

\subsection{Periods $\dis T=1\,/\,n$}\label{SecExistenceNT1}
                                                                                                  
\begin{theorem}
Properties (\ref{TMhyp}) and (\ref{H1H2}) are assumed. If $n\,T=1$, $n\in \mathbb{N}$,   
 and the source $s(.)$ belongs to $H^1_\sharp$, then there exists a unique $T$-periodic solution $y \in  H^1_\sharp $ to (\ref{ToyModel}), which satisfies the non autonomous ODE
\begin{equation}
(1+c)\,y^{'}(t)+f(y(t))+g(y(t))=s(t).
\label{ThNT1ODE}
\end{equation}
\label{ThNT1}
Furthermore, if $c=1$, we can take the source $s(.)\in L^2_\sharp$.
\end{theorem}

In other words, if the delay $1$ is a period of the source, there exists only one periodic solution for the (NDDE) with the same period as the source. 

\begin{proof}
Considering a 1-periodic solution $y$ to (\ref{ToyModel}) gives (\ref{ThNT1ODE}). If $c=+1$, then
\begin{equation}
y^{'}(t)+h(y(t))=\frac{\textstyle s(t)}{\textstyle 2},
\label{ThNT1proofA}
\end{equation}
with $h=(f+g)\,/\,2$. Under the assumptions (\ref{H1H2}), $h$ satisfies $h(0)=0$ and $h^{'}(0)>0$. The existence, uniqueness and asymptotic stability of a periodic solution to (\ref{ThNT1proofA}) follows accordingly; see for instance the proof of theorem 5-1 in \cite{JuLo12}. If $c=-1$, then (\ref{H1H2}) ensures that $s$ belongs to the range of $h$, and hence $y(t)=h^{-1}(s(t))$.
\end{proof}

Theorem \ref{ThNT1} does not ensure the asymptotic stability of the solution to the NDDE (\ref{ToyModel}). For this purpose, an hypothesis of stability of  small solution must be added, as in proposition \ref{PropPeriode}.

%---------------------------------------------------------------------

\subsection{General periods}\label{SecExistenceGeneral}

Unlike the special cases investigated in section \ref{SecExistenceNT1}, the results obtained in this section involve a much larger set of periods $T$ but for small sources. The next theorem is also valid for almost periodic sources: it suffices to replace the periodic spectrum $\{k\,\omega,\; k \in \mathbb{Z}\}$ by the almost periodic spectrum 
$\{\omega_k,\; k \in \mathbb{Z}\}$ in (\ref{dth1}), (\ref{ConditionDiophant}) below. Quasi-periodicity can occur, for instance, when the source is made of multiple incommensurable periods. Now, we state the main result of existence. 

\begin{theorem}
Let us consider the nonlinear non autonomous NDDE (\ref{ToyModel}) with assumptions (\ref{TMhyp})-(\ref{H1}). Based on the integer $\Omega_k$ (\ref{OmegaN}), the subsequence $\phi$ is chosen such that   
\begin{equation} \label{dth1}
\left|k\,\omega-\Omega(\phi(k))\right|=\min_n\left|k\,\omega-\Omega(n)\right|.
\end{equation}
The source is $T$-periodic, small and belongs to $H^1_{\sharp}$. If the condition 
\begin{equation}
\liminf_k \left|d-\Omega(\phi(k))\times \left(k\,\omega-\Omega(\phi(k))\right)\right|>0
\label{ConditionDiophant}
\end{equation}
is satisfied, then there exists a unique periodic solution $y\in H^1_{\sharp}$. 
\label{ThDiophant}
\end{theorem}

\begin{proof}
Theorem \ref{ThSmallDiv} and condition (\ref{ConditionDiophant}) imply that $\inf_k | h(i\,\omega_k)|>0$. The operator $s(t)\rightarrow y(t)$ in the linearized NDDE (\ref{NDDElinear}) is thus continuous. The implicit function theorem can therefore be applied, ensuring the existence of a unique solution in a neighborhood of the origin as in \cite{JuLo12}. 
\end{proof}

Note that $\phi(k)\sim \frac{k}{T}\rightarrow \pm \infty$. The condition (\ref{ConditionDiophant}) means that there are no small divisors in the linearized NDDE (\ref{NDDElinearS}). Now, the key issue is to know when (\ref{Diophant}) or (\ref{ConditionDiophant}) are fulfilled. An important case is obtained for rational values of the period, as stated in the next theorem.

\begin{theorem}[Rational period]
The same notations and assumptions are used as in theorem \ref{ThDiophant}. If $T \in \mathbb{Q}$,  then the condition (\ref{ConditionDiophant}) is fulfilled and the conclusion of 
theorem \ref{ThDiophant} holds. 

Furthermore, if $c=1$, let us write $T=p\,/\,q$ with $gcd(p,\,q)=1$. If $p$ is odd, then a small source in $L^2_\sharp$ yields a unique periodic solution $y\in H^1_{\sharp}$.
\label{ThRational}
\end{theorem}

In other words, no small divisors occur when the period of the source is rational. Moreover, when  $c=+1$, periods in the form $T=\frac{2\,m+1}{q}$ ($m,\,q \in \mathbb{Z},\; q \neq 0$) provide a smoothing effect.  

\begin{proof}
Let $v_k$ be $\dis\min_{n\in \mathbb{Z}}\left|k\,\omega-\Omega(n)\right|=\left|k\,\omega-\Omega(\phi(k))\right|$. The condition (\ref{ConditionDiophant}) is fulfilled if there exists $\delta>0$ such that $\dis Z_k=\left|d-\Omega(\phi(k))\,v_k\right| \geq \delta>0$ for $|k|$ large enough. Before examining the existence of $\delta$, we recall that $d>0$ and $\Omega(\phi(k))\sim k\,\omega$. 

The distance to a real subset $A \subset \mathbb {R}$ is denoted: $\dist(x,\,A)=\dis\inf_{a \in A}|x-a|$. From the definitions of $\omega$ (\ref{TMHyp4}) and $\Omega$ (\ref{OmegaN}), it follows: 
\begin{eqnarray*}
v_k &=& 2\,\pi\,\min_{n\in \mathbb{Z}} \left|\frac{k}{T}-\frac{1+c}{4}-n\right|,\\
    &=& 2\,\pi\, \dist\left(\frac{k}{T}-\frac{1+c}{4},\,\mathbb{Z}\right),\\
    &=& 2\,\pi\, \dist\left(k\,\frac{q}{p}-\frac{1+c}{4},\,\mathbb{Z}\right).
\end{eqnarray*}
Since $\dist (x\,\mathbb{Z}) \leq \frac{1}{2}$ for all $x$, it follows $0 \leq v_k \leq \pi$. Moreover, periodicity of $(v_k)$ implies that the sequence $(v_k)$ takes at most $p$ values. Let $0 \leq \underline{v}=\dis \inf_{k \in \mathbb{Z}} v_k=\dis \min_{0\leq k<p} v_k$. If $\underline{v}=0$, then we take $\delta=d$. If $\underline{v}>0$, then any positive $\delta $ works since $Z_k \rightarrow + \infty$. Consequently, the condition (\ref{ConditionDiophant}) is always satisfied when $T$ is rational, ensuring the existence of a unique periodic solution. 
 
Let us now examine deeply the cases where $\underline{v}=0$. This equality occurs in two cases:
\begin{itemize}
\item If $c=-1$, then $\underline{v}=0$. It follows from $v_{k \times p}=2\,\pi\,\dist\left(k\,q,\,\mathbb{Z}\right)=0$.  
\item If $c=+1$, then $\underline{v}=0$ if and only if $p \in 2\,\mathbb{Z}$. This necessary and sufficient condition follows from   
$v_{k}=2\,\pi\,\dist\left(k\,\frac{q}{p}-\frac{1}{2},\,\mathbb{Z}\right)=0$ for some $k\neq 0$. It amounts to say that $k\,\frac{q}{p}-\frac{1}{2}=m$ for some $m \in \mathbb{Z}$, i.e. $q\,/\,p=(2\,m + 1)\,/\,(2\,k)$, and hence $p$ has to be even. 
\end{itemize}  
It proves that $\underline{v}=0$ iff $c=+1$ and $p$ is odd. In this case, the characteristic function ${\cal H}(\lambda)$ in (\ref{H}) satisfies ${\cal H}(\lambda)\equiv\lambda\,(1+c\,e^{-\lambda})+(a+b\,e^{-\lambda})\sim\lambda\,(1+e^{-\lambda})$ when $|k| \rightarrow + \infty$. Since the linearized operator is a Fourier multiplier by the characteristic function, $|{\cal H}(\lambda_k)|\sim|k|$ ensures that the linearized equation is smoothing from $L^2_\sharp$ into $H^1_\sharp$.
\end{proof}

%---------------------------------------------------------------------

\subsection{About the Diophantine condition}\label{SecDiophant}

Up to now, we have proven that rational periods satisfy the condition (\ref{Diophant}), ensuring the validity of theorem \ref{ThDiophant}. Are there other sets of periods in a similar case ? We have not resolved this question. Here, we only propose few results on this topic. More precisely, we study the complementary of this set, when the small divisors problem occurs. To estimate this complementary set, the following definition is introduced.

\begin{definition}[Diophantine condition ${\cal D}$]
Let $D=d\,/\,(4\,\pi^2)>0$, with $d=f^{'}(0)-c\,g^{'}(0)>0$. The real number $\theta\in\mathbb{R}^+$ satisfies the condition ${\cal D}$ if and only if there exists an infinite number of integers $(n,\,k)$ such that
\begin{equation}
\theta=\frac{n}{k}+\frac{1+c}{4\,k}+\frac{D}{\theta}\,\frac{1}{k^2}+ o\left(\frac{1}{k^2}\right).
\label{Diophant}
\end{equation}
\label{DefDiophant}
\end{definition}
The next lemma states that the condition (\ref{Diophant}) is exactly the converse of the condition (\ref{ConditionDiophant}).

\begin{lemma} 
The period $T$ satisfies (\ref{ConditionDiophant}) if and only if $\theta=\frac{1}{T}$ does not satisfy (\ref{Diophant}). 
\end{lemma}
\begin{proof}
To prove the equivalence, let us rewrite the condition in theorem \ref{ThSmallDiv} to have a small divisors problem. In this case, $\omega$ is such that there exists $\phi: \mathbb{Z} \mapsto \mathbb{Z}$ satisfying 
\begin{eqnarray*}   
|\phi(k)|                                   & \rightarrow & \infty, \\ 
k\,\omega - \Omega(\phi(k))                 & \rightarrow & 0,      \\
\Omega(\phi(k))\,(k\,\omega-\Omega(\phi(k)) & \rightarrow & d,  
\end{eqnarray*}
when $|k|\rightarrow \infty$. This is equivalent to say that there exists an infinite number of $(k,\,n)\in \mathbb{Z}\times \mathbb{Z}$ such that 
\begin{equation}
\frac{1}{T}=\frac{n}{k}+\frac{1+c}{4\,k}+\frac{D\,T}{k^2}+ o\left(\frac{1}{k^2}\right),
\label{BadSet0}
\end{equation}
with $D =\frac{d}{4 \pi^2}$: consequently, $1\,/\,T$ does not satisfy the Diophantine condition (\ref{Diophant}).
\end{proof}

The size of the set of numbers satisfying the Diophantine condition (\ref{Diophant}) is unknown. To advance on this question, we consider $C\in \mathbb{R}$, and we define the set $E(C)$ as follows: $\theta\in E(C)$ if there exists an infinite number of $(k,\,n) \in \mathbb{Z}\times \mathbb{Z}$ such that 
\begin{eqnarray} 
\theta=\frac{n}{k}+\frac{C}{k^2}+ o\left(\frac{1}{k^2}\right).
\label{BadSet1}
\end{eqnarray}
$E(C)$ is close to the set of reals satisfying (\ref{BadSet0}). Indeed, $\frac{n}{k}+\frac{1+c}{4\,k} =\frac{4\,n+(1+c)}{4\,k} \in \mathbb{Q}$ and $D\,T$ is replaced by $C$ in (\ref{BadSet0}). Moreover, $E(C)$ is bigger than the set of numbers satisfying ${\cal D}$, since the condition (\ref{Diophant}) is more restrictive than (\ref{BadSet1}).

\begin{proposition}
Depending on the value of $C$, we get the following alternative:
\begin{itemize}
\item if $C=0$, then $\mathbb{Q} \subset E(C)$. Moreover, almost all real numbers belong to $E(C)$, but $E(C) \neq \mathbb{R}$; 
\item if $C\neq 0$, then $\mathbb{Q} \cap E(C)=\emptyset$.
\end{itemize}
\label{PropEC}
\end{proposition}
\begin{proof}  
We examine successively the two cases:
\begin{itemize} 
\item Case $C=0$. Clearly all rational numbers are in $E(0)$.

We recall a Khintchine's theorem on Diophantine approximation \cite{Khintchine64,Dodson09}. If $\phi$ is a non-increasing function, $\phi: \;  \mathbb{N}^* \mapsto ]0,+\infty[$ and if $\sum_{n>0} \phi(n)=+\infty$ then, for almost all real numbers, there are infinitely  many rational $\dis p\,/\,q$ such that $\dis\left|x-p\,/\,q\right|<\phi(q)\,/\,q$. 

It suffices to take $\dis \phi(n-1)=1\,/\,(n\ln(n))$ to conclude that almost all real numbers are in $E(0)$. Indeed, we have 
$$
\dis \frac{\phi(q)}{q} \sim \frac{1 }{q^2 \ln q}=o\left(\frac{1}{q^2}\right).
$$
Notice that $E(0) \neq \mathbb{R}$. For instance, the golden ratio $\dis (1+\sqrt{5})\,/\,2$ does not belong to $E(0)$, see \cite{Khintchine64}.  

\item Case $C \neq 0$: let $\theta=p\,/\,q$ be in $E(C)$. From (\ref{BadSet1}), it follows that
$k\,p-n\,q=\frac{C\,q}{k}+o(k^{-1})$, and hence $k\,p-n\,q$ is a convergent sequence of integers towards 0. Such a sequence of integers is equal to $0$ for $k$ large enough. But this is impossible since $C \neq 0$. Consequently, $\mathbb{Q} \cap E(C)=\emptyset$.
\end{itemize}
\end{proof}

%---------------------------------------------------------------------
%---------------------------------------------------------------------
 
\section{Conclusion}\label{SecConclu}

The main results of this article are summed up as follows.
\begin{itemize}
\item theorem \ref{ThZero}, corollary \ref{CoroD-f} and corollary \ref{CoroD-fg}: asymptotic stability of constant solutions in the nonlinear case is proven for a constant forcing. The results are valid whatever the amplitude of the source;
\item proposition \ref{PropPeriode}: asymptotic stability of a periodic source under a periodic non-constant forcing and in the nonlinear case is obtained under two assumptions (stability and small amplitude of the source);
\item theorems \ref{ThStabDiagP1} and \ref{ThStabDiagM1}: stability diagram is analyzed whatever the coefficients of the linearized NDDE;
\item theorem \ref{ThNT1}: existence of a periodic solution to the nonlinear NDDE is proven whatever the amplitude of the source but for particular values of the period. Theorem \ref{ThDiophant} considers a much larger set of periods, but under the assumption of small source and a Diophantine condition. Theorem \ref{ThRational} states that rational periods satisfy this condition.
\end{itemize}
Deeper analysis is required about the set of real numbers not satisfying (\ref{Diophant}). Proposition \ref{PropEC} is a first step in this direction, but the measure of the set involved remains an open question.

%---------------------------------------------------------------------

\section*{Acknowledgments}

This research was supported by the GdR 2501, CNRS, France. The authors are grateful to Thomas Erneux for fruitful discussions about neutral equations and essential instability in Peyresq summer school 2011. We wish also to thank G\'erard Iooss for his insights about small divisors.

%---------------------------------------------------------------------

\end{document}